\documentclass[12pt,a4paper,oneside]{amsart}

\ifdefined\SMART
\usepackage[paperwidth=12cm,paperheight=24cm,left=5mm,right=5mm,top=10mm,bottom=5mm]{geometry}
\else
\calclayout
\fi

\usepackage{amssymb,latexsym,amsfonts,textcomp,fancyhdr,calc,graphicx}
\usepackage{pdfpages,amsfonts,amsthm,amssymb,latexsym,graphicx,leftidx,fancyhdr}
\usepackage{amsmath,amstext,amsthm,amssymb,amsxtra}
\usepackage[allcolors=blue,allbordercolors=blue,pdfborderstyle={/S/U/W 1}]{hyperref}
\usepackage[ocgcolorlinks]{ocgx2}
\usepackage{dsfont}
\usepackage[framemethod=tikz]{mdframed}
\usepackage{asymptote}
\usepackage{enumerate}
\usepackage{changepage}

\usepackage[normalem]{ulem}
\usepackage{cite}

\usepackage{txfonts,pxfonts,tikz} 
\usepackage{tgschola}
\usepackage[T1]{fontenc}

\usepackage{mathtools}

\usepackage{bbm}

\theoremstyle{plain}
\newtheorem{theorem}{Theorem}
\newtheorem{lemma}{Lemma}

\theoremstyle{definition}
\newtheorem{definition}{Definition}

\newtheorem{case[theorem]}{Case}

\newcommand{\beql}[1]{\begin{equation}\label{#1}}
\newcommand{\eeq}{\end{equation}}

\newcommand{\Abs}[1]{{\left|{#1}\right|}}

\newcommand{\Set}[1]{{\left\{{#1}\right\}}}

\newcommand{\one}{{\bf 1}}

\newcommand{\RR}{{\mathbb R}}

\newcommand{\ZZ}{{\mathbb Z}}

\newcommand{\QQ}{{\mathbb Q}}

\newcommand{\dens}{{\rm dens\,}}

\newcommand{\rank}{{\rm rank\,}}

\newcommand{\vol}{{\rm vol\,}}

\newcommand{\wt}[1]{\widetilde{#1}}

\newcounter{rem}
\newcounter{step}
\newcounter{mysec}
\newcounter{mysubsec}[mysec]
\begin{document}
\sloppy

\title{Bounded common fundamental domains for two lattices}

\author{Sigrid Grepstad}
\address{\href{https://www.ntnu.edu/imf}{Department of Mathematical Sciences}, Norwegian University of Science and Technology (NTNU), NO-7491, Trondheim, Norway.}
\email{sigrid.grepstad@ntnu.no}

\author{Mihail N. Kolountzakis}
\address{\href{http://math.uoc.gr/en/index.html}{Department of Mathematics and Applied Mathematics}, University of Crete,\\Voutes Campus, 70013 Heraklion, Greece,\newline and \newline \href{https://ics.forth.gr/}{Institute of Computer Science}, Foundation of Research and Technology Hellas, N. Plastira 100, Vassilika Vouton, 700 13, Heraklion, Greece}
\email{kolount@uoc.gr}

\makeatletter
\@namedef{subjclassname@2020}{\textup{2020} Mathematics Subject Classification}
\makeatother
\thanks{Sigrid Grepstad is supported by Grant 334466 of the Research Council of Norway.}


\begin{abstract}
We prove that for any two lattices $L, M \subseteq \RR^d$ of the same volume there exists a measurable, bounded, common fundamental domain of them. In other words, there exists a bounded measurable set $E \subseteq \RR^d$ such that $E$ tiles $\RR^d$ when translated by $L$ or by $M$. A consequence of this is that the indicator function of $E$ forms a Weyl--Heisenberg (Gabor) orthogonal basis of $L^2(\RR^d)$ when translated by $L$ and modulated by $M^*$, the dual lattice of {$M$}.
\end{abstract}

\date{\today}

\maketitle

\tableofcontents

\newcommand{\trn}{\Set{0}^m\times\RR^n}
\newcommand{\ztrn}{\ZZ^m\times\RR^n}

\section{Introduction}\label{s:intro}

\subsection{The Steinhaus tiling problem}\label{ss:steinhaus}

A question of Steinhaus from the 1950s \cite{moser1981research,sierpinski1958probleme} asks if there is a subset $E$ of the plane $\RR^2$ such that $E$ tiles the plane when translated by $R_\theta\ZZ^2$, for any value of $\theta$. Here $R_\theta$ denotes the $2\times 2$ matrix which rotates the plane by the angle $\theta$ around the origin. Equivalently we are seeking a set $E$ such that $R_\theta {E}$ tiles the plane when translated by $\ZZ^2$, for any $\theta$.

For a set $E\subseteq\RR^d$ to tile $\RR^d$ when translated by the set $T \subseteq \RR^d$ we mean that the $T$-translates of $E$ partition $\RR^d$. If the set $T$ happens to be a subgroup of $\RR^d$ this is the same as demanding that $E$ contains exactly one element from each coset of $T$ in $\RR^d$. Clearly this definition of tiling makes sense in any abelian group.

The Steinhaus tiling problem bifurcated from the 1980s into two forms:
\begin{itemize}
\item the original, \textit{set-theoretic} formulation where nothing else is expected from the set $E$ but to contain one point from each coset of $R_\theta\ZZ^2$, and this for any $\theta$, and
\item the \textit{measurable} formulation, where the set $E$ is expected to be Lebesgue measurable but, in return, the tiling is demanded {almost} everywhere: for any $\theta$ we only ask that
\begin{equation}\label{tiling-meas}
\sum_{n \in R_\theta\ZZ^2} \one_E(x-n) = 1,\ \ \text{ for almost all } x \in \RR^2.
\end{equation}
\end{itemize}
{We} should add that the problem makes sense in $\RR^d$, $d>2$, as well, where we are seeking a {set} $E$ that tiles simultaneously with all linear transformations of $\ZZ^d$ by an orthogonal matrix (though we must admit that sensible forms of this problem may be stated even with smaller groups). 

The set-theoretic question in the plane ($d=2$) was settled in a major result by Jackson and Mauldin \cite{jackson2002lattice,jackson2002sets,jackson2003survey} who proved the existence of such a set $E$ in the plane. 

The measurable question is still open in the plane. There have been many partial results, almost all of which are of the form ``if a measurable Steinhaus set $E$ exists it must be large near infinity''. For example it is known \cite{beck1989lattice,kolountzakis1996problem} that such a set cannot be bounded. The best result so far in this direction is that in \cite{kolountzakis1999steinhaus} where it is shown that
$$
\int_E \Abs{x}^\alpha\,dx = +\infty \text{ for } \alpha > 46/27.
$$

In an interesting lack of symmetry between the set-theoretic and measurable developments it is now known  \cite{kolountzakis1999steinhaus,kolountzakis2002steinhaus} that there are no measurable Steinhaus sets in dimensions $d>2$ but it is still unknown if there are ``set-theoretic'' Steinhaus sets for $d>2$.

The interested reader should consult the references in  \cite{kolountzakis2017measurable} as well as the most recent paper \cite{kiss2024solutions}, for results on many variations of the Steinhaus question.

\subsection{Common fundamental domains for finitely many lattices}\label{ss:multi}

A \textit{fundamental domain} for an abelian group $H$ within an abelian group $G$ is a subset of $G$ that contains exactly one element from every coset of $H$ in $G$. So, the Steinhaus tiling problem for the plane asks for a common fundamental domain for all groups $R_\theta \ZZ^2$ inside $\RR^2$, for $\theta \in [0, 2\pi)$.

From now on, we focus on the measurable version of the problem where we only ask $E$ to satisfy the tiling equation \eqref{tiling-meas} almost everywhere.

A sensible relaxation of the Steinhaus problem is to look for a common fundamental domain of only a finite family of lattices
\begin{equation}\label{collection}
{L_1, \ldots, L_n \in \RR^d}.
\end{equation}
Any measurable fundamental domain of a lattice has volume equal to the determinant (also called volume) of the lattice. Hence, we must require that all {$L_1, \ldots, L_n$} have the same volume.

In \cite{kolountzakis1997multi} it was proved that if the dual lattices of the collection \eqref{collection} have a direct sum
$$
{L_1^*+\cdots+L_n^*
}$$
then we can find a measurable common fundamental domain for \eqref{collection}. And it was shown in \cite{han2001lattice} that for the case of two lattices only no condition is necessary: Any two lattices of the same volume in $\RR^d$ have a measurable common fundamental domain. (See also \cite{kolountzakis2022functions} for several similar questions.)

In both \cite{kolountzakis1997multi} and \cite{han2001lattice} the constructed fundamental domains are generally unbounded. Since then, it has been an open problem whether two lattices of the same volume in $\RR^d$ have a measurable bounded common fundamental domain in $\RR^d$. This question we answer in this paper:
 
\begin{theorem}\label{th:main}
Suppose $L,M$ are lattices in $\RR^d$ of the same volume.
Then there is a bounded measurable $\Omega \subseteq \RR^d$ which tiles with both $L$ and $M$.
\end{theorem}

The important technical breakthrough arises in the special case below when $L$ and $M$ have a direct sum. This is made possible using one of the main results in \cite{EGKL2025}.

\begin{theorem}\label{th:dense}
If $L, M \subseteq \RR^d$ are lattices of the same volume and $\overline{L+M}=\RR^d$ then there is a bounded, measurable $E \subseteq \RR^d$ such that $L \oplus E = M \oplus E = \RR^d$ are both tilings. 
\end{theorem}
We again stress that when we say that the direct-sum equality $L \oplus E = \RR^d$ or $M \oplus E = \RR^d$ is a tiling, we mean that this holds, except possibly on a set of measure zero.

\subsection{An application to Weyl--Heisenberg orthogonal bases}\label{ss:gabor}

In \cite{han2001lattice} the existence of a measurable common fundamental domain for two lattices is used to show that whenever $K, L$ are two lattices in $\RR^d$ with $\det L \cdot \det K = 1$ then there exists a Gabor (or Weyl-Heisenberg) orthogonal basis of $\RR^d$ with translation lattice $L$ and modulation lattice $K$. In other words, there exists a function $g \in L^2(\RR^d)$ such that the collection of time-frequency translates
$$
e^{2\pi i \ell \cdot x} g(x-k),\ \ \ell \in L, k \in K,
$$
is an orthogonal basis of $L^2(\RR^d)$. In their proof the function $g$ is precisely the indicator function of a measurable common fundamental domain of the lattices $K$ and $L^*$. Thus our Theorem \ref{th:main} implies that this \textit{window} function $g$ may be chosen to be of compact support, a possibly significant property, since it offers the advantage of localization.

\subsection{Some notation}\label{ss:notation}

A \textit{lattice} is a discrete subgroup of $\RR^n$ which linearly spans $\RR^n$. The \textit{rank} of a subgroup of $\RR^n$ is the dimension of its linear span. Thus a lattice is a discrete subgroup of $\RR^n$ of full rank, equal to $n$. We denote by $\vol L$ or $\det L$ the volume of any fundamental domain of the lattice $L$, and by $\dens L$ the lattice density $1/\vol L$. If $L$ is a discrete subgroup of $\RR^d$ of rank smaller than $d$ we still write $\vol L$ or $\det L$ to denote the volume of the fundamental domain in the $\RR$-linear space $L$ spans.

Any lattice $L \subseteq \RR^n$ is equal to $A\ZZ^n$ where $A$ is a non-singular $n\times n$ matrix. This matrix $A$ is not unique, but can be formed by taking as its columns any $\ZZ$-basis of $L$.
The \textit{dual lattice} of $L$ is defined by
$$
L^* = \Set{x\in\RR^n: x\cdot \ell \in \ZZ \text{ for all } \ell \in L}
$$
and it can be seen that $L^* = A^{-\top}\ZZ^n$.

When we write $A \oplus B$ for two sets $A, B$ in an additive group we mean that all sums $a+b$, with $a \in A, b \in B$, are distinct. In this case we say the sum $A+B$ is \textit{direct} or that $A+B$ is a \textit{tiling}.

\noindent{\bf Plan.} We prove Theorem \ref{th:dense} first in \S \ref{s:dense-sum} and use it to prove then Theorem \ref{th:main} in \S \ref{s:general}.

\section{Bounded common fundamental domains when the sum is dense}\label{s:dense-sum}

The proof of Theorem \ref{th:dense} relies on certain results from the theory of so-called cut-and-project sets in $\RR^d$. We therefore give a brief description of this point set construction, introducing necessary notation and terminology.

A discrete point set $\Lambda$ in $\RR^d$ is called a \emph{Delone set} if it is both uniformly discrete and relatively dense, meaning there exist radii $r,R>0$ such that any ball of radius $r$ contains at most one point of $\Lambda$, and any ball of radius $R$ contains at least one point of $\Lambda$. If $\Lambda$ additionally satisfies
\begin{equation*}
    \Lambda - \Lambda \subseteq \Lambda + F,
\end{equation*}
for some finite set $F$ in $\RR^d$, then we say that $\Lambda$ is a \emph{Meyer set}.

A cut-and-project set, or \emph{model set}, is constructed from a lattice $\Gamma \subset \mathbb{R}^m \times \mathbb{R}^n$ and a \emph{window set} $W\subset \mathbb{R}^n$ by taking the projection into $\mathbb{R}^m$ of those lattice points whose projection into $\mathbb{R}^n$ is contained in $W$. Denoting the projections from $\mathbb{R}^m \times \mathbb{R}^n$ onto $\mathbb{R}^m$ and $\mathbb{R}^n$ by $p_1$ and $p_2$, respectively, we assume that $p_1 |_{\Gamma}$ is injective, and that the image $p_2(\Gamma)$ is dense in $\mathbb{R}^n$, and denote by $\Lambda_W=\Lambda(\Gamma,W)$ the model set
\begin{equation*}
\Lambda (\Gamma, W) = \left\{ p_1(\gamma) \, : \, \gamma \in \Gamma , \, p_2(\gamma) \in W \right\} .
\end{equation*}

If the boundary $\partial W$ of the window $W$ has Lebesgue measure zero, then the model set $\Lambda_W$ is called \emph{regular}. In this case, the point set $\Lambda_W$ in $\mathbb{R}^m$ has a number of desirable properties. One can show that $\Lambda_W$ is a Meyer set with well-defined density 
\begin{equation*}
    \dens \Lambda_W = \frac{|W|}{\det \Gamma} = \Abs{W} \cdot \dens\Gamma.
\end{equation*}
Moreover, if the model set is either generic (meaning that $p_2(\Gamma) \cap \partial W = \emptyset$) or if the window $W$ is half-open as defined in \cite[Definition 2.2]{P2000}, then $\Lambda_W$ is \emph{repetitive}. Repetitivity is the crystal-like quality that every finite configuration appearing in $\Lambda$ will reappear infinitely often, see e.g.\ \cite[Property 2]{P2000} for a precise definition.

The cut-and-project construction is well-studied in the field of aperiodic order, and in the last 30 years there have been several results on when a model set (or more generally a Delone set) is at bounded distance from a lattice \cite{DO1991, FG2018, Lac1992}. We say that two point sets $\Lambda$ and $\Lambda'$ in $\mathbb{R}^n$ are \textit{bounded distance equivalent} (or, at bounded distance from each other) if there exists a bijection $\varphi \, : \, \Lambda \to \Lambda'$ and a constant $C>0$ such that
\[\|\varphi (\lambda)-\lambda \| < C\]
for all $\lambda \in \Lambda$.

\noindent{\bf Facts:}
\begin{enumerate}
\item Bounded distance equivalence is an equivalence relation.
\item If a Delone set $\Lambda$ in $\mathbb{R}^d$ has a well-defined density and is bounded distance equivalent to a lattice $L$ in $\mathbb{R}^d$, then $\dens  \Lambda = \dens L$.
\item Any two lattices $L$ and $M$ in $\mathbb{R}^d$ of equal density are necessarily {at bounded distance from each other} (\cite[Theorem 5.2]{DO1990}, \cite[Theorem 1]{DO1991}, \cite[\S 3.2]{kolountzakis1997multi}).
\end{enumerate}

The proof of Theorem \ref{th:dense} relies on the following result from \cite{DO1991} on model sets with parallelotope windows, as well as a more recent result from \cite{EGKL2025} connecting bounded distance equivalence and equidecomposability (Theorem \ref{thm:equidecomp} below).
\begin{theorem}{\cite[Theorem 3.1]{DO1991}}\label{thm:do}
Let $\Gamma$ be a lattice in $\mathbb{R}^m \times \mathbb{R}^n$. If $W \subset \mathbb{R}^n$ is a parallelotope
\[W= \left\{ \sum_{k=1}^n t_k v_k \, : \, 0 \leq t_k < 1 \right\}\]
spanned by $n$ linearly independent vectors in $p_2(\Gamma)$, then the model set $\Lambda(\Gamma, W)$ is at bounded distance to a lattice in $\mathbb{R}^m$.
\end{theorem}
We say that two sets $S$ and $S'$ in $\mathbb{R}^m$ are \emph{equidecomposable} if $S$ can be partitioned into finitely many subsets which can be rearranged by translations to form a partition of $S'$. Given a subgroup $G \subset \mathbb{R}^m$ we will use the term \emph{$G$-equidecomposable} to mean that we allow only translations in $G$ for this rearrangement.

Theorem \ref{thm:do} above can be extended to hold for any reasonably well-behaved fundamental domain of a sublattice in $p_2(\Gamma)$ by the following result of Frettl\"{o}h and Garber.
\begin{theorem}{\cite[Theorem 6.1]{FG2018}}\label{thm:fg}
Let $\Lambda$ and $\Lambda'$ be two model sets constructed from the same lattice $\Gamma$ but with different windows $W$ and $W'$, respectively. If the windows $W$ and $W'$ are $p_2(\Gamma)$-equidecomposable, then $\Lambda$ and $\Lambda'$ are bounded distance equivalent. 
\end{theorem}
It turns out that for regular model sets, a converse of Theorem \ref{thm:fg} can be established if we relax the equidecomposability condition to ignore sets of measure zero.
\begin{definition} Let $G$ be a group of translations in $\mathbb{R}^n$. We say that two measurable sets $S$ and $S'$ in $\mathbb{R}^n$ of equal Lebesgue measure are \emph{$G$-equidecomposable up to measure zero} if there exists a partition of $S$ into finitely many measurable subsets $S_1, \ldots , S_N$, and a set of vectors $g_1, \ldots , g_N \in G$, such that $S'$ and $\bigcup_{j=1}^N (S_j+g_j)$ differ at most on a set of measure zero.
\end{definition}
\begin{theorem}{\cite[Theorem 4.1]{EGKL2025}}\label{thm:equidecomp}
Let $\Gamma \subset \mathbb{R}^m \times \mathbb{R}^n$ be a lattice and let $W$ and $W'$ be bounded, measurable sets in $\mathbb{R}^n$ {where both $\partial W$ and $\partial W'$ have measure zero and $|W|=|W'|$}. If the model sets $\Lambda_W = \Lambda(\Gamma, W)$ and $\Lambda_{W'}=\Lambda(\Gamma, W')$ are bounded distance equivalent, then the windows $W$ and $W'$ are $p_2(\Gamma)$-equidecomposable up to measure zero.
\end{theorem}

We are now equipped to prove Theorem \ref{th:dense}.
\begin{proof}[Proof of Theorem \ref{th:dense}]
By abuse of notation let $L=L\mathbb{Z}^d$ and $M=M\mathbb{Z}^d$, where $L$ and $M$ are $d\times d$ non-singular matrices. Let $\Omega_L$ be the half-open parallelotope spanned by the columns of $L$ and $\Omega_M$ be the half-open parallelotope spanned by the columns of $M$. Then $\Omega_L$ and $\Omega_M$ are fundamental domains of the lattices $L$ and $M$, respectively. Since $L$ and $M$ are assumed to have equal volumes, we have $|\Omega_M| = |\Omega_L|$.

We now construct a lattice $\Gamma \subset \mathbb{R}^d \times \mathbb{R}^d$ (where $\Gamma$ again denotes both the lattice itself and its matrix representation) by letting
\begin{equation*}
\Gamma = 
\left(
\begin{array}{c c c c c}
\\
& & K & & \\
\\
 \hline
 \\
& L &  & M & \\ 
 \\
\end{array}
\right),
\end{equation*}
where $K$ is chosen to be a $d \times 2d$ matrix which acts as an injective map on $\mathbb{Z}^{2d}$. With the cut-and-projection construction in mind, note that this ensures that the projection $p_1$ is injective when restricted to the lattice $\Gamma$. This requirement for $K = [K_1 | K_2]$ is a generic condition on $K$ (say, the matrices $K$ failing it have measure $0$) hence the matrix $K_1$ can also be assumed to be nonsingular, in which case we can apply the block-determinant formula and obtain
\begin{equation}\label{block-det}
\det\Gamma = \det K_1 \det(M-L K_1^{-1} K_2).
\end{equation}
In order to guarantee the non-singularity of $\Gamma$ it is enough to pick the matrix $K_2$ to have very small elements. Since $M$ is nonsingular and the set of nonsingular matrices is open it follows that the second factor on the right-hand side of  \eqref{block-det} is also non-zero and thus $\Gamma$ is a proper lattice in $\RR^{2d}$. Finally, since $\overline{L+M}=\mathbb{R}^d$ by assumption, we know that $p_2(\Gamma)$ is dense in $\mathbb{R}^d$.

We now consider the two model sets 
\begin{equation*}
\Lambda_L= \Lambda(\Gamma, \Omega_L) = \left\{ p_1(\gamma) \, : \, \gamma \in \Gamma , p_2(\gamma) \in \Omega_L \right\}
\end{equation*}
and
\begin{equation*}
\Lambda_M = \Lambda(\Gamma, \Omega_M) = \left\{ p_1(\gamma) \, : \, \gamma \in \Gamma , p_2(\gamma) \in \Omega_M \right\}.
\end{equation*}
Since $p_2(\Gamma) = L + M$, we see that both $\Omega_L$ and $\Omega_M$ are windows spanned by $d$ linearly independent vectors in $p_2(\Gamma)$. Thus by Theorem \ref{thm:do}, both $\Lambda_L$ and $\Lambda_M$ are bounded distance equivalent to a lattice. Moreover, by assumption we have $|\Omega_L|=|\Omega_M|$, so $\dens \Lambda_L= \dens \Lambda_M$. This implies that the model sets $\Lambda_L$ and $\Lambda_M$ are bounded distance equivalent to lattices of equal density, and thus also {at bounded distance from each other}. We thus conclude from Theorem \ref{thm:equidecomp} that we can find a partition of $\Omega_L$ into subsets $S_1, \ldots , S_N$ and elements $\gamma_1 , \ldots , \gamma_N \in \Gamma$ such that 
\begin{equation}\label{eq:fdm}
\Omega_M = \bigcup_{i=1}^N \underbrace{\left(S_i + p_2(\gamma_i) \right)}_{S_i'} = \bigcup_{i=1}^N S_i',
\end{equation}
where we understand this equality to hold up to measure zero.

Finally, we observe that 
\begin{equation*}
p_2(\gamma_i) = \ell_i + m_i, 
\end{equation*}
for every $i=1, \ldots , N$, where $\ell_i \in L$ and $m_i \in M$. It follows that
\begin{equation*}
E=\bigcup_{i=1}^N (S_i' - m_i) = \bigcup_{i=1}^N (S_i + \ell_i)
\end{equation*}
is a fundamental domain for both $M$ and $L$ by \eqref{eq:fdm} and the fact that $(S_i)_{i=1}^N$ is a partition of $\Omega_L$. We thus have
\begin{equation*}
L \oplus E = M \oplus E = \mathbb{R}^d,
\end{equation*}
for a bounded measurable set $E \subset \RR^d$.
\end{proof}

\section{Bounded common fundamental domains in the general case}\label{s:general}

In this section we prove Theorem \ref{th:main}.

\begin{lemma}\label{lm:ranks}
Suppose $L \subseteq \ztrn$ is a lattice in $\RR^d$, where $d=m+n$. Then
$$
L_2 = L \cap \trn
$$
has rank $n$.
\end{lemma}

\begin{proof}
Suppose $\rank L_2 = k < n$ and let $u_1, \ldots, u_k \in \trn$ be a $\ZZ$-basis of $L_2$. Let also $u_{k+1},\ldots,u_d$ be an extension of this $\ZZ$-basis to a $\ZZ$-basis of $L$. This extension always exists \cite[Corollary 3, p.\ 14]{cassels1996introduction}.

It follows that there are $g_j \in \ZZ^m\times\Set{0}^n$ and $r_j \in \trn$, for $j=1,\ldots,d-k$, such that
$$
u_{k+j} = g_j + r_j,\ \ j=1,\ldots d-k.
$$
Since $m<d-k$ there are $n_j \in \ZZ$, not all $0$, such that $\sum_{j=1}^{d-k} n_j  g_j = 0$. This implies that $0 \neq \sum_{j=1}^{d-k} n_j u_{k+j} \in \trn$, hence this sum belongs to $L_2$, a contradiction, since $u_1,\ldots,u_d$ are linearly independent and $L_2$ is generated by $u_1, \ldots, u_k$.

\end{proof}

\begin{lemma}{\cite[Theorem 4.3]{ore1958}}\label{lm:two}
Suppose $G_1, G_2$ are subgroups of the group $G$ of the same, finite index $k$. Then there are $g_1, \ldots, g_k \in G$ which are simultaneously a complete set of coset representatives of $G_1$ and $G_2$ in $G$. In other words
$$
G_1 + \Set{g_1, \ldots, g_k} = G_2 + \Set{g_1, \ldots, g_k} = G
$$
are both tilings. 
\end{lemma}






The proof of Theorem \ref{th:main} follows.
{We use induction on $d=m+n$ in section \ref{ss:general} below. For $d=1$ Theorem \ref{th:main} is true since there is only one lattice of a given volume in dimension 1.}

The closed subgroups of $\RR^d$ are, up to a non-singular linear transformation, of the form
\begin{equation}\label{close-subgroup}
\ztrn
\end{equation}
where $m+n=d$, where $m=0, 1, \ldots, d$ \cite[Theorem 9.11]{hewitt2012abstract}. Thus we may assume that $\overline{L+M} = \ztrn$ for some such decomposition $d=m+n$. Next we observe that it is enough to find a bounded common fundamental domain $\Omega'$ of $L,M$ in $\ztrn$ which is measurable in $\ztrn$. Then we can take $\Omega = \Omega'+[0,1]^m\times\Set{0}^n$. 

\subsection{Case $m=0$.}

This is Theorem \ref{th:dense}: $L+M$ is dense in $\RR^d$ and they have the same volume, so there is a bounded common tile for them. 

\subsection{Case $m=d$.}

We have $L+M = \ZZ^d$. The lattices have the same volume, hence the same index in $\ZZ^d$. By Lemma \ref{lm:two} there exists a finite set $F \subseteq \ZZ^d$ such that $L+F = M + F = \ZZ^d$ are tilings. 

\subsection{General case: $0 < m < d$.}\label{ss:general} {Notice that $n<d$, so we can assume Theorem \ref{th:main} to be true for lattices in dimension $n$.} Define the discrete subgroups of $\trn$
$$
L_2 = (\trn) \cap L \ \text{ and } \ M_2 = (\trn) \cap M.
$$
By Lemma \ref{lm:ranks} the groups $L_2, M_2$ have rank $n$.
Write
$$
L = L_1\oplus L_2,\ \ \ M=M_1\oplus M_2,
$$
where $L_1, M_1$ are discrete subgroups of $\ztrn$ of rank $m$.
Since the sums are direct it follows that the points of $L_1$ are all different mod $\trn$ and so are all points of $M_1$.

Abusing notation we can write $L = L\ZZ^d$, $M=M\ZZ^d$, where $L, M$ are $d\times d$ non-singular matrices. The columns of these matrices can be any basis of the lattices so we choose the first $m$ to be a basis of $L_1$ (resp.\ $M_1$) and the last $n$ to be a basis of $L_2$ (resp.\ $M_2$). The matrices $L, M$ are now lower block triangular
$$
L = \begin{pmatrix}
L_1 & 0 \\ * & L_2
\end{pmatrix},\ \ \ 
M = \begin{pmatrix}
M_1 & 0 \\ * & M_2
\end{pmatrix},
$$
where the $m\times m$ matrices $L_1, M_1$ have integer entries since these entries represent the first $m$ coordinates of a basis of $L_1, M_1 \subseteq \ZZ^m\times\RR^n$. It follows that
$$
\det L = \det L_1 \cdot \det L_2 \ \ \text{ and }\ \ 
\det M = \det M_1 \cdot \det M_2.
$$
Since $\det L = \det M$ and $\det L_1, \det M_1 \in \ZZ$ we have that
\begin{equation}\label{ratios}
\frac{\det L_2}{\det M_2} = \frac{\det M_1}{\det L_1} \in \QQ.
\end{equation}
All determinants in \eqref{ratios} are non-zero and can be assumed positive.

Therefore we have the group indices
\begin{equation}\label{indices}
	[\ztrn : L_1\oplus\trn] = \det L_1 \ \ \text{ and }\ \ 
	[\ztrn : M_1\oplus\trn] = \det M_1.
\end{equation}
{(Notice that in \eqref{indices} above $L_1$ appears both as a lattice and as a matrix, and similarly for $M_1$.)}

We also have that
\begin{equation}\label{l1m1}
	L_1+M_1+\trn = \ztrn,
\end{equation}
since the left hand side is a subgroup of the right hand side. If it were a proper subgroup then we could not have $\overline{L+M}=\ztrn$.

\subsubsection{A simple case. Not strictly necessary for the rest, but easier.}

If, besides $\det L = \det M$, we also have that the ratios in \eqref{ratios} are equal to 1, so that $\det L_i = \det M_i$, $i=1,2$, then we can find{, invoking the inductive hypothesis since $n<d$,} a bounded common tile $E'$ of $L_2$ and $M_2$ in $\trn$:
\begin{equation}\label{l2m2eprime}
L_2 \oplus E' = M_2 \oplus E' = \trn.
\end{equation}
From \eqref{indices} the groups $L_1 \oplus \trn$ and $M_1 \oplus \trn$ have the same finite index $\det L_1 = \det M_1$ in the group $\ztrn$, hence, from Lemma \ref{lm:two} we can find a common, finite tile $F$ of them in $\ztrn$:
\begin{equation}\label{l1m1f}
L_1\oplus\trn \oplus F = M_1 \oplus \trn \oplus F = \ztrn.
\end{equation}
From \eqref{l2m2eprime} and \eqref{l1m1f} we obtain
$$
L_1\oplus (L_2 \oplus E') \oplus F =
M_1\oplus (M_2 \oplus E') \oplus F = \ztrn,
$$
so with $E = F \oplus E'$ we obtain the tilings
$$
L \oplus E = M \oplus E = \ztrn.
$$
This concludes the proof of this simple case.

\ 

In general the ratios in \eqref{ratios} are not necessarily 1.
Take now $L_2'$ and $M_2'$ to be super-lattices of $L_2$ and $M_2$ in $\trn$ such that
\begin{equation}\label{indices-new}
[L_2' : L_2] = \det M_1\ \ \text{ and } \ \ [M_2' : M_2] = \det L_1.
\end{equation}
It follows from \eqref{ratios} that
\begin{equation}\label{ratios1}
\det L_2' = \frac{\det L_2}{\det M_1} =
 \frac{\det M_2}{\det L_1} = \det M_2'.
\end{equation}
{Since $L_2', M_2' \subseteq \trn$ and $n<d$ it follows by induction} that there is a bounded common tile $E'$ of $L_2'$ and $M_2'$ in $\trn$. We have
$$
\Abs{E'} = \det L_2'  = \det M_2 '.
$$
Let the finite sets $J_2\subseteq L_2'$ and $K_2\subseteq M_2'$ be such that
$$
L_2' = L_2 \oplus J_2 \ \text{ and }\ M_2' = M_2 \oplus K_2
$$
are both tilings, so that it follows from \eqref{indices-new} that $\Abs{J_2} = \det M_1$ and $\Abs{K_2} = \det L_1$.

Since $L_1 + M_1 + \trn = \ztrn$ from \eqref{l1m1} we can find finite sets $J_1 \subseteq L_1$ of size $\Abs{J_1} = \det M_1$ and $K_1 \subseteq M_1$ of size $\Abs{K_1} = \det L_1$ (these sizes follow from \eqref{indices}) such that
\begin{equation}\label{plane-tiling-1}
K_1 \oplus L_1 \oplus \trn = \ztrn
\end{equation}
and
\begin{equation}\label{plane-tiling-2}
J_1 \oplus M_1 \oplus \trn = \ztrn.
\end{equation}

Since $\Abs{J_1} = \Abs{J_2}$ and $\Abs{K_1} = \Abs{K_2}$ we can find bijections
$$
\phi:K_1 \to K_2,\ \ \ \psi:J_1 \to J_2.
$$
Define the sum
\begin{equation}\label{tile-E}
E = \Set{x+y+\phi(x)+\psi(y): x\in K_1,\ y \in J_1} \ \oplus\  E' \ \subseteq \  \ztrn.
\end{equation}
The fact that the sum in \eqref{tile-E} is direct is a byproduct of the proof that follows in which we show that the set $E$ is a common tile for the lattices $L$ and $M$.

For reasons of symmetry we need only verify that
$$
L \oplus E = L_1\oplus L_2 \oplus E = \ZZ^m \times \RR^n
$$
is a tiling.

We first show that this is a packing.
Let $\ell=\ell_1+\ell_2$ and $\wt\ell=\wt\ell_1+\wt\ell_2$ be elements of $L = L_1\oplus L_2$ and assume that the two translates $\ell+E$ and $\wt\ell+E$ overlap on positive measure. This means that there are
$$
x, \wt x \in K_1,\ \ y, \wt y \in J_1
$$
such that
$$
\ell_1 + \ell_2 + x + y + \phi(x) + \psi(y) + E'
\ \text{ and }\ 
\wt\ell_1 + \wt\ell_2 + \wt x + \wt y + \phi(\wt x) + \psi(\wt y) + E'
$$
overlap on positive measure. These can be rewritten as
$$
\underbrace{\ell_1 + y}_{\in L_1} + x + \underbrace{\ell_2 + \phi(x) + \psi(y) + E'}_{\subseteq \trn}
$$
and
$$ 
\underbrace{\wt\ell_1 + \wt y}_{\in L_1} + \wt x + \underbrace{\wt\ell_2 + \phi(\wt x) + \psi(\wt y) + E'}_{\subseteq \trn}.
$$
Since $x, \wt x \in K_1$ we get, because of tiling condition \eqref{plane-tiling-1}, that
\begin{equation}\label{first-conclusion}
\ell_1+y =\wt\ell_1+\wt y \ \ \text{ and }\ \  x = \wt x,
\end{equation}
which of course implies that $\phi(x) = \phi(\wt x)$.
Thus the translates
$$
\ell_2 + \psi(y) + E' \ \ \text{ and }\ \ \wt\ell_2 + \wt\psi(y) + E'
$$
overlap on positive measure. But $\ell_2+\psi(y), \wt\ell_2+\wt\psi(\wt y) \in L_2'$ and $L_2'\oplus E' = L_2 \oplus J_2 \oplus E'$ are tilings, so we get $\ell_2 = \wt\ell_2$ and $\psi(y) = \psi(\wt y)$. The last equation implies $y = \wt y$ since $\psi$ is a bijection. Finally from \eqref{first-conclusion} we obtain $\ell_1 = \wt\ell_1$.

We have shown that the translates of $E'$ that participate in the definition \eqref{tile-E} of the tile $E$ are all non-overlapping and, therefore,
\begin{align}\label{E-volume}
\Abs{E} &= \Abs{E'} \cdot \Abs{K_1} \cdot \Abs{J_1} \\
 &= \det L_2' \cdot \det L_1 \cdot \det M_1 \nonumber\\
 &= \det L_1 \cdot \det L_2' \cdot\Abs{J_2}\nonumber \\
 &= \det L_1 \cdot \det L_2 \nonumber \\
 &= \det L. \nonumber
\end{align} 
We also showed that $L+E$ is a packing. Since the arrangement $L+E$ is periodic it follows that $L\oplus E = \ztrn$ is a tiling. By symmetry so is $M \oplus E = \ztrn$. 

The final bounded common tile $\Omega$ of $L$ and $M$ in $\RR^d$ is then given by
$$
\Omega = E + [0,1)^m\times\Set{0}^n.
$$

\bibliographystyle{alpha}
\bibliography{mk-bibliography.bib}

\end{document}